\documentclass[12pt]{article}
\usepackage{amsmath,amssymb,amsfonts,amsthm,url,xspace,graphicx,psfrag}
\usepackage[pdftitle={The min mean-weight cycle in a random network},
            pdfauthor={Claire Mathieu and David B. Wilson},
            bookmarks=true,bookmarksopen=true,bookmarksopenlevel=2]{hyperref}
\usepackage[all]{hypcap}

\newcommand{\sref}[1]{\hyperref[#1]{\S~\ref*{#1}}}
\newcommand{\aref}[1]{\hyperref[#1]{Appendix~\ref*{#1}}}
\newcommand{\lref}[1]{\hyperref[#1]{Lemma~\ref*{#1}}}
\newcommand{\tref}[1]{\hyperref[#1]{Theorem~\ref*{#1}}}
\newcommand{\cref}[1]{\hyperref[#1]{Corollary~\ref*{#1}}}
\newcommand{\fref}[1]{\hyperref[#1]{Figure~\ref*{#1}}}
\newcommand{\pref}[1]{\hyperref[#1]{Proposition~\ref*{#1}}}
\newcommand{\tblref}[1]{\hyperref[#1]{Table~\ref*{#1}}}

\def\clap#1{\hbox to 0pt{\hss#1\hss}}

\newcommand{\weight}{\operatorname{weight}}
\newcommand{\length}{\operatorname{length}}
\newcommand{\Var}{\operatorname{Var}}

\newcommand{\MRhref}[2]{\href{http://www.ams.org/mathscinet-getitem?mr=#1}{MR#2}}
\makeatletter
\def\@strippedMR{}
\def\@scanforMR#1#2#3\endscan{%
  \ifx#1M\ifx#2R\def\@strippedMR{#3}%
  \else\def\@strippedMR{#1#2#3}%
  \fi\fi}
\def\@rst #1 #2other{#1}
\newcommand\MR[1]{\relax\ifhmode\unskip\spacefactor3000 \space\fi
  \MRhref{\expandafter\@rst #1 other}{#1}}
\newcommand\MRs[1]{\relax\ifhmode\unskip\spacefactor3000 \space\fi
  \@scanforMR#1\endscan
  \MRhref{\@strippedMR}{\@strippedMR}}
\makeatother

\numberwithin{equation}{section}
\newtheorem{theorem}{Theorem}
\numberwithin{theorem}{section}

\newtheorem{lemma}[theorem]{Lemma}

\newcommand{\E}{\mathbb E}
\renewcommand{\Pr}{\mathbb P}
\newcommand{\eps}{\varepsilon}

\begin{document}

\author{\begin{tabular}{c}
\href{http://www.cs.brown.edu/~claire/}{Claire Mathieu}\thanks{Work partly funded by NSF AF 0964037, and partly done while visiting Microsoft Research.}\\[-4pt]
\small CNRS (ENS) and Brown
\end{tabular} \and
\begin{tabular}{c}
\href{http://dbwilson.com}{David B. Wilson}\\[-4pt]
\small Microsoft Research
\end{tabular}
}
\title{\vspace{-48pt}The min mean-weight cycle in a random network
\footnotetext{\textit{2010 Mathematics Subject Classification.}  05C80, 68Q87.}}
\date{}
\maketitle
\begin{abstract}
  The mean weight of a cycle in an edge-weighted graph is the sum of
  the cycle's edge weights divided by the cycle's length.   We study the minimum mean-weight cycle on the complete
  graph on $n$ vertices, with random i.i.d.\ edge weights drawn from
  an exponential distribution with mean~$1$.  We show that the
  probability of the min mean weight being at most $c/n$ tends to a
  limiting function of $c$ which is analytic for $c\leq 1/e$,
  discontinuous at $c=1/e$, and equal to $1$ for $c>1/e$.  We further
  show that if the min mean weight is $\leq 1/(en)$, then the length
  of the relevant cycle is $\Theta_p(1)$ (i.e., it has a limiting
  probability distribution which does not scale with $n$), but that if
  the min mean weight is $>1/(en)$, then the relevant cycle almost
  always has mean weight $(1+o(1))/(en)$ and length at least
  $(2/\pi^2-o(1)) \log^2 n \log\log n$.
\end{abstract}

\section{Introduction}

Many combinatorial optimization problems have been studied when the input is a complete (directed or undirected) graph with independent random weights on the edges.  This line of work has been active since the mid-1980s for problems such as the minimum spanning tree \cite{frieze,frieze-mcdiarmid}, shortest path \cite{frieze-grimmet,hooghiemstra, janson99,hooghiemstra2,hooghiemstra3},  traveling salesman path~\cite{frieze2}, minimum weight perfect matching (the assignment problem)~\cite{aldous,MR2036492,nair}, spanners~\cite{chebolu}, and Steiner tree~\cite{BGRS,AFW}.  In this paper, we study the \textit{minimum mean-weight cycle}.

Given a directed graph with arc weights, the minimum mean-weight cycle
problem is that of finding a cycle with minimum mean weight. The mean
weight of a cycle is the ratio between its total weight and its number
of arcs.  The min mean-weight cycle problem, and the closely related minimum ratio cycle
problem (where each arc has a cost and a transit time, and the mean ratio of a cycle is the total cost divided by the total transit time), have applications in areas ranging from discrete event
systems and computer-aided design to graph theory; see
Dasdan~\cite{dasdan} for a detailed discussion and references.  An
experimental study of various algorithms for min mean cycle can be
found in~\cite{georgiadis}, including experiments on random
graphs. An algorithm by Young, Tarjan, and Orlin~\cite{YTO}
emerges as particularly efficient.  Their algorithm is based on the
parametric shortest path problem, which is the problem of finding
shortest paths in graph where the edge costs are of the form
$w_{i,j}+\lambda$, where each $w_{i,j}$ is constant and $\lambda$ is a
parameter that varies.  This problem is well-defined when $\lambda$
is at least
$$-\min_{\text{cycle }C} \frac{\sum_{ij\in C} w_{i,j}}{|C|},$$
but when $\lambda$ is below this value there is a negative cycle, so
the problem becomes ill-defined.  The authors of \cite{YTO}
conjectured that their algorithm is faster on average than in the
worst case, by a factor of $n$; analyzing the structure of the min
mean cycle is an intermediate step towards that conjecture.

In this paper, we study the min mean-weight cycle in the complete
graph on $n$ vertices, with random i.i.d.\ edge weights drawn from an
exponential distribution with mean~$1$, so that $\Pr[ w_e>x ]=e^{-x}$.
We do this for both the directed complete graph, which is relevant to
the experiments of Young, Tarjan, and Orlin \cite{YTO} and subsequent
experiments, and for the undirected complete graph, so that we can
more readily compare our results to earlier work on cycles in the
random graph $G_{n,p}$ \cite{janson,MR1001395}.

The min \textit{max}-weight cycle has been studied
by Janson \cite{janson} and others \cite{MR1001395}. 
One way to instantiate the random graph $G_{n,p}$ is to
start with the undirected complete graph with i.i.d.\ exponential edge weights, and
put each edge in
$G_{n,p}$ if its weight is smaller than~$\log 1/(1-p)$
(or if we instead use weights that are uniform in~$[0,1]$, 
the edge is included if its weight is smaller than~$p$).
 As the parameter $p$ is
increased from $0$ to $1$, the first cycle to appear 
is the min \textit{max}-weight cycle.  Janson
\cite{janson} gives formulas for when that cycle occurs (i.e., its
max-weight), and for its length distribution: the
probability that the min max-weight cycle has max weight less than $c/n$ tends as $n\to\infty$ to a
continuous function of~$c$, which is analytic and increases from $0$
to $1$ as $c$ increases from $0$ to $1$, is non-analytic but continuous at $c=1$, and
equals $1$ for $c>1$ (see \fref{cdf-4}). The limiting length distribution
(see \tblref{tbl:p_k}) is completely
supported on finite values (i.e., which don't grow with $n$), but this
distribution has a fat tail which gives it an infinite expected value.
(For finite $n$, the expected length is order $n^{1/6}$
\cite{MR1001395}.)
\enlargethispage{\baselineskip}

We find that the min \textit{mean}-weight cycle has a qualitatively
different behavior: the probability that the min mean-weight cycle
has mean weight at most $c/n$ tends as $n\to\infty$ to a function of $c$ which is
piece-wise analytic, but which is \textit{discontinuous\/} at $c=1/e$
(see \fref{cdf-4}).
More precisely,
the mean weight of the min mean-weight cycle is with
constant probability within an interval $(1/e,1/e+o(1))/n$, where the $o(1)$ term
goes to $0$ as $n\to\infty$.
Furthermore, the limiting length distribution of the min mean-weight
cycle is not supported on finite values.  
In other words, probability that
the min mean-weight cycle has length
$k$ tends to a positive limiting value $p_k$, but $\sum_k p_k < 1$
(see \tblref{tbl:p_k}).
What this means is that with constant
probability the cycle has length order $1$, and with constant
probability the cycle has a length which is a function of $n$ tending
to infinity.

It is natural to ask what this function of $n$ is.  The behavior of
the long cycles is complicated, and we do not conjecture a value for
the true answer.  The best that we could prove is that the length of
the min mean-weight cycle is almost always either $O(1)$ or else at
least $(2/\pi^2-o(1)) \log^2 n \log\log n$.

\begin{figure}[t!]
\psfrag{c}[c][c][0.8]{\footnotesize $c$}
\centerline{\psfrag{label}[c][c]{\footnotesize min max-weight cycle, undirected graph}\includegraphics[width=0.46\textwidth]{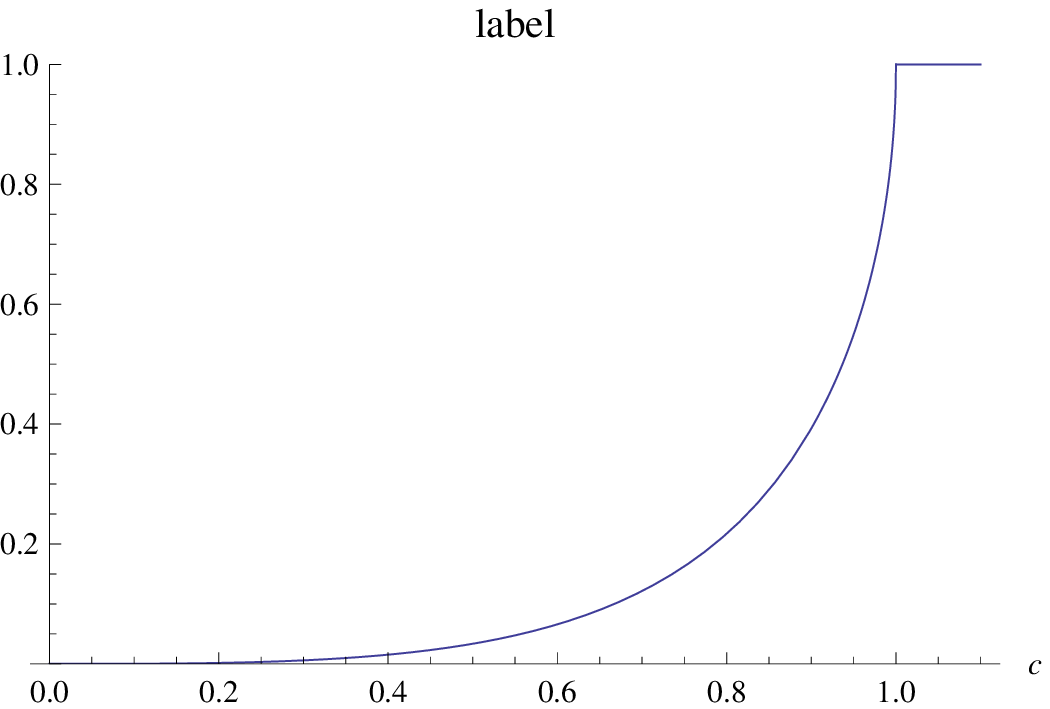} \hfill \psfrag{label}[c][c]{\footnotesize min mean-weight cycle, undirected graph}\includegraphics[width=0.46\textwidth]{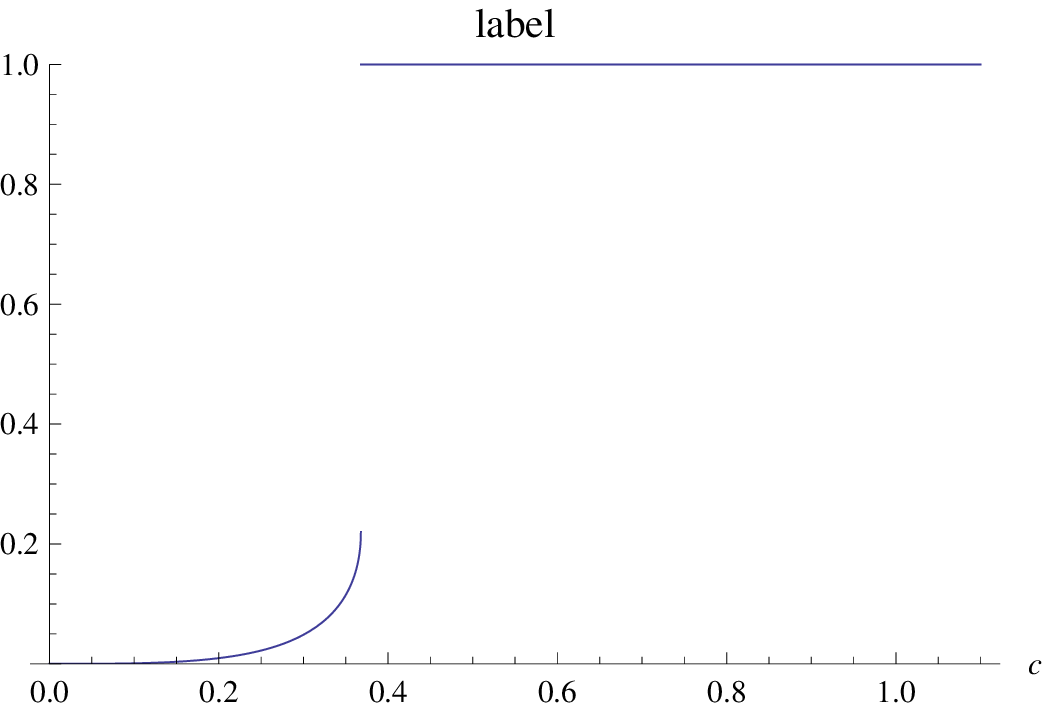}}
\vspace{15pt}
\centerline{\psfrag{label}[c][c]{\footnotesize min max-weight cycle, directed graph}\includegraphics[width=0.46\textwidth]{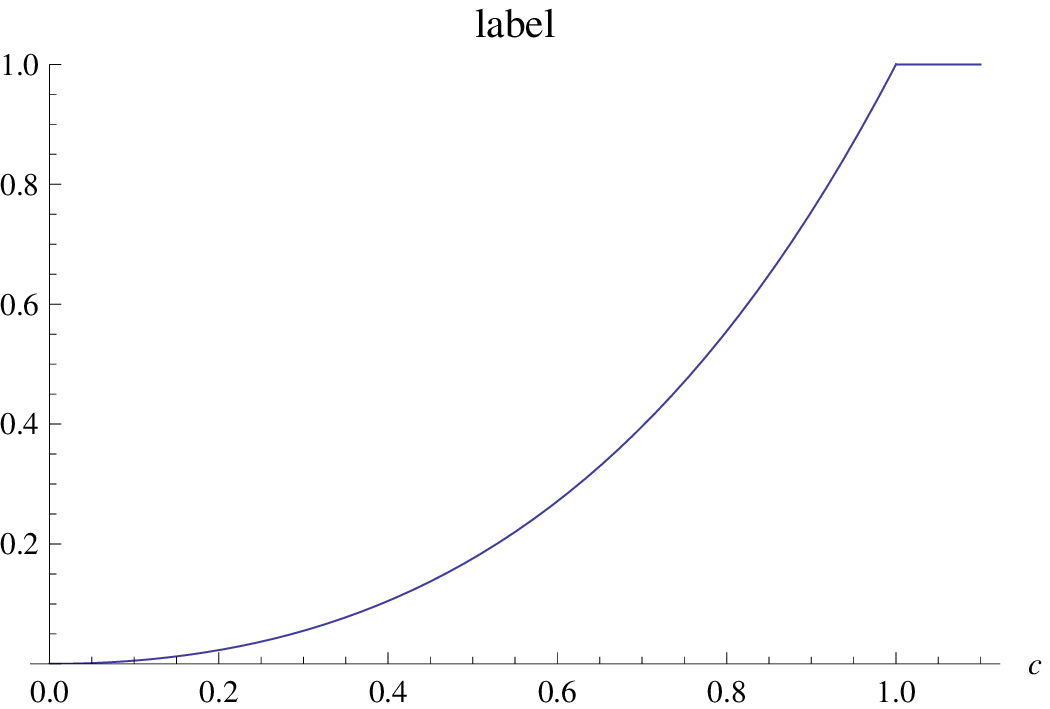} \hfill \psfrag{label}[c][c]{\footnotesize min mean-weight cycle, directed graph}\includegraphics[width=0.46\textwidth]{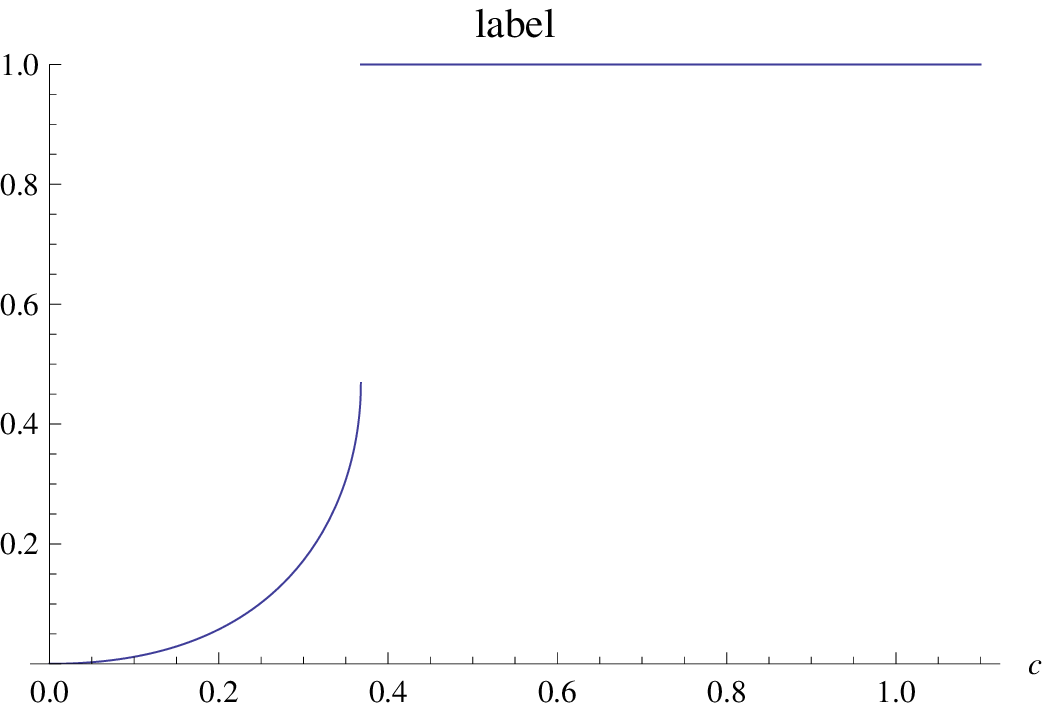}}
\caption{
The probability that the min max/mean-weight cycle in the undirected/directed complete graph has max/mean-weight less than $c/n$.
(The upper left panel was computed by Janson \cite{janson}.)  For the min max-weight cycle, the function is non-analytic at $c=1$, but is continuous.  For the min mean-weight cycle, the function is discontinuous at $c=1/e$.
 \label{cdf-4}}
\end{figure}

In the related problem of finding the maximum length path whose mean
weight is at most $c/n$, Aldous \cite{aldous:1/e} found that there is
a transition point at $c=1/e$, where for fixed $c<1/e$ the length is
$o(n)$, and for fixed $c>1/e$ the length is order $n$.  Recently Ding
\cite{ding} studied the behavior of this path length when $c$ is at or
near $1/e$, and proved that the length exhibits a transition at $c=1/e +
\Theta(1/\log^2 n)$, with unspecified constants.  By comparison, we
prove that with probability $1-o(1)$ the min mean-weight cycle has mean
weight at most $(1/e+(\pi^2+o(1))/(2 e \log^2 n))/n$, but we do not know
if the $O(1/\log^2 n)$ correction term is sharp.

Whether the complete graph is directed or undirected will affect the
length distribution and the max/mean-weight of the min max/mean-weight
cycle, but each of the qualitative behaviors discussed above is
unaffected by whether the graph is directed or undirected (as shown in
\fref{cdf-4} and \tblref{tbl:p_k}), though the exponent
characterizing the fatness of the tail of the length distribution does
change for the min max-weight cycle.

\begin{table}[t!]
$$
\begin{array}{ccccc}
 & \text{undirected graph}
 & \text{directed graph}
 & \text{undirected graph}
 & \text{directed graph}
\\
 & \text{min max cycle}
 & \text{min max cycle}
 & \text{min mean cycle}
 & \text{min mean cycle}
\\
\hline
 p_{2} & \text{---} & 0.281718 & \text{---} & 0.116616 \\
 p_{3} & 0.121608 & 0.154845 & 0.035248 & 0.061750 \\
 p_{4} & 0.084915 & 0.098900 & 0.022796 & 0.039132 \\
 p_{5} & 0.063827 & 0.068937 & 0.016229 & 0.027417 \\
 p_{6} & 0.050329 & 0.050915 & 0.012283 & 0.020485 \\
 p_{7} & 0.041047 & 0.039195 & 0.009701 & 0.016005 \\
 p_{8} & 0.034331 & 0.031129 & 0.007905 & 0.012923 \\
 p_{9} & 0.029280 & 0.025334 & 0.006598 & 0.010701 \\
 p_{10} & 0.025365 & 0.021027 & 0.005613 & 0.009039 \\
 p_{100} & 0.000921 & 0.000264 & 0.000165 & 0.000238 \\
\\
 p_{k} &
 \displaystyle\frac{1}{2}\! \int_0^1\! c^{k-1} \sqrt{1{-}c}\, e^{\frac{c}{2}+\frac{c^2}{4}} \, dc &
 \displaystyle\int_0^1\!\! c^{k-1} (1{-}c) e^c \, dc &
 \displaystyle\frac{k^k}{2 k!}\! \int_0^{\frac{1}{e}} \frac{c^{k-\frac{1}{2}} e^{\frac{c}{2}+\frac{c^2}{2}}}{\sqrt{T(c)}} \, dc &
 \displaystyle\frac{k^k}{k!}\! \int_0^{\frac{1}{e}} \frac{c^k e^c}{T(c)} \, dc \\
\\
 p_{k}
 &
 (1{+}o(\!1\!)\!)\displaystyle\frac{e^{3/4} \sqrt{\pi }}{4 k^{3/2}} &
 (1{+}o(\!1\!)\!)\displaystyle\frac{e}{k^2} &
 (1{+}o(\!1\!)\!)\displaystyle\frac{e^{-\frac{1}{2}+\frac{1}{2 e}+\frac{1}{2 e^2}}}{\sqrt{8 \pi} k^{3/2}} &
 (1{+}o(\!1\!)\!)\displaystyle\frac{e^{-1+\frac{1}{e}}}{\sqrt{2 \pi} k^{3/2} } \\
\\
 p_{k}
 &
 \displaystyle\frac{0.938071\dots{+}o(\!1\!)}{k^{3/2}} &
 \displaystyle\frac{2.71828\dots{+}o(\!1\!)}{k^2} &
 \displaystyle\frac{0.155598\dots{+}o(\!1\!)}{k^{3/2}} &
 \displaystyle\frac{0.212023\dots{+}o(\!1\!)}{k^{3/2}} \\
\\
 \sum _{k} p_{k} & 1 & 1 & 1-e^{-\frac{1}{2}+\frac{1}{2 e}+\frac{1}{2 e^2}} & 1-e^{-1+\frac{1}{e}} \\
  &  &  & =0.219946\dots & =0.468536\dots
\end{array}
$$
\caption{
  The limiting length distribution of the min max/mean-weight cycle.  The leftmost column is due to Janson \cite{janson}.  Here $T(c)= \sum_{k=1}^\infty k^{k-1} c^k/k!$ is the ``tree function''.  
For the min max-weight cycle, the length distribution is supported on finite values, while for the min mean-weight cycle, a constant fraction of the probability mass ($1-\sum_k p_k$) drifts off to infinity.  The size of the jumps at the discontinuities in Figure~\ref{cdf-4} is $1-\sum_k p_k$.
}
\label{tbl:p_k}
\end{table}

\enlargethispage{12pt}
We call a cycle \textit{$c$-light\/} if its mean weight is $<c/n$.  We
start with an elementary calculation of the expected number of
$c$-light cycles of length $k$.  Then we show that for $c\leq 1/e$,
the set of light cycles is well approximated by a Poisson process with
intensity given by the first-moment computation.  For $c>1/e$, the
number of $c$-light cycles diverges.  Given this Poisson
approximation, it is straightforward to do the computations
illustrated in \fref{cdf-4} and \tblref{tbl:p_k}.  A key difference
between the min mean-weight cycle and min max-weight cycle is that the
expected number of $c$-light cycles is finite at the critical value
$c=1/e$, while for the max-weight cycles, the expected number of light
cycles diverges at the critical value of~$c$.  As we will explain, the
finite expected number of light cycles at the critical value of $c$ is
what leads to the discontinuity in the curves in \fref{cdf-4} and it
is also why $\sum_k p_k < 1$.  With probability tending to $1-\sum_k
p_k$ the min mean-weight cycle is long (has length tending to infinity
with $n$) and has mean weight $(1/e+o(1))/n$; analyzing its length is
difficult because the Poisson approximation breaks down in this
regime, but we bound it below by $(2/\pi^2-o(1))\log^2 n \log\log n$.

\section{Review of the tree function}

Because it plays a key role in the formulas for min mean-weight
cycles in the subcritical regime, i.e., for weight $<1/(en)$, we
briefly review the {\em tree function} and the closely related Lambert
$W$ function.  The tree function $T$ is the exponential generating
function for rooted spanning trees.  Recalling Cayley's formula that
there are $k^{k-1}$ rooted spanning trees on $k$ nodes, we have
\[ T(z) = \sum_{k=1}^\infty k^{k-1} \frac{z^k}{k!}.
\]
From Stirling's formula
\[ \sqrt{2\pi k}\, \frac{k^k}{e^k} \leq k! \leq \sqrt{2\pi k}\, \frac{k^k}{e^k} e^{1/(12k)}, \]
this sum
converges when $|z|\leq 1/e$.  Using techniques from the theory of
generating functions, one can see that
\begin{equation}
 T(z) = z \,e^{T(z)} \label{eq:tree-rec}
\end{equation}
 (see e.g., \cite[Proposition~5.3.1]{stanley}).
It is straightforward to check that
$ T(1/e) = 1$.
Near this critical point, using~\eqref{eq:tree-rec}, one can deduce
\begin{equation}\label{eq:treefunctionatcritical}
 T(\textstyle\frac{1-\delta}{e}) = 1-\sqrt{2\delta} + O(\delta).
 \end{equation}
The Lambert $W$ function is defined by the equation
$$ z = W(z) \, e^{W(z)}. $$
This is a multivalued function, but the principal branch is defined so that
$W(z) = -T(-z)$ when $|z|\leq 1/e$, and by analytic continuation elsewhere.
The tree function figures prominently in the analysis of random graphs
near the phase transition (see e.g., \cite{JKLP}), and the Lambert $W$
function is an important function in applied mathematics; for further
background see \cite{MR1414285}.

\section{The expected number of light cycles}

Given $c>0$, say that a directed or undirected $k$-cycle~$C$ or $k$-path $P$ is
\textit{$c$-light\/} if its mean weight $w(C)/k$ is at most $c/n$.

\begin{lemma} \label{lemma:prelim}
With exponential edge weights,
if $0\leq c_1\leq c_2$ then
$$\Pr[\text{$k$-cycle or $k$-path is $c_2$-light but not $c_1$-light}]
\begin{cases}
\displaystyle\sim \frac{k^k}{k!}\frac{c_2^k-c_1^k}{n^k} & \text{if $k=o(n)$}\\[9pt]
\displaystyle \leq \frac{k^k}{k!}\frac{c_2^k-c_1^k}{n^k}
& \text{for any $k$.}
\end{cases}$$
\end{lemma}

\begin{proof}
The weight $w(C)$ of a $k$-cycle or $k$-path $C$ is distributed as
the sum of $k$ independent exponential random variables, that is,
according to the Gamma distribution with shape parameter $k$,
which has density function
\begin{equation}\label{eq:density}
\phi: x\mapsto e^{-x} {x^{k-1}}/{\Gamma(k)},
 \end{equation}
where $\Gamma$ is the gamma function, which is
$\Gamma(k)=(k-1)!$ for positive integers $k$.  Thus
$\Pr[c_1 k/n < w(C)\leq c_2 k/n]=\int_{c_1 k/n}^{c_2 k/n}\phi(x)\,dx.$
Now $e^{-k/n}\leq 1$ and when $k=o(n)$ we have $e^{-k/n}=1-o(1)$.
\end{proof}

\begin{lemma}\label{lemma:prelim2}
Let $N_k$ denote the number of directed $k$-cycles in the complete graph on $n$ vertices. For $k\geq2$ we have
$$N_k \begin{cases}
\displaystyle\sim \frac{n^k}{k} & \text{if $k=o(\sqrt{n})$}\\[9pt]
\displaystyle\leq \frac{n^k}{k} &\text{for any $k$.}
\end{cases} $$
(For $k\geq 3$, the number of undirected $k$-cycles is of course $\frac12 N_k$.)
\end{lemma}
\begin{proof}
$$\frac{kN_k}{n^k}=\frac{n(n-1)\cdots (n-k+1)}{n^k}=
  \prod_{i=0}^{k-1}\left(1-\frac{i}{n}\right)=\exp\left[-\frac{k^2}{2n} + O\left(\frac{k}{n}\right) +O\left(\frac{k^3}{n^2}\right)+\cdots \right]$$
which is $1-o(1)$  when $k=o(\sqrt{n})$, and at most $1$ in all cases.
\end{proof}

\begin{theorem}\label{thm:EZ}
  For the directed complete graph, let $Z^{(k)}_c$ denote the number of directed $k$-cycles with mean
  weight less than $c/n$, and $Z_c=\sum_k Z^{(k)}_c$ denote the total
  number of $c$-light directed cycles.  If $0\leq c_1\leq c_2$ then
\begin{equation}\label{eq:zk}
 \E[Z^{(k)}_{c_2}-Z^{(k)}_{c_1}] =
\begin{cases}
\displaystyle \sim \frac{(c_2^k-c_1^k) k^k}{k\times k!} & \text{if $k=o(\sqrt{n})$}\\[9pt]
\displaystyle \leq \frac{(c_2^k-c_1^k) k^k}{k\times k!} & \text{for any $k$,}
\end{cases}
\end{equation}
and
\begin{equation}
\lim_{n\rightarrow\infty}\E[Z_c]=
\begin{cases}
\displaystyle\sum_{2\leq k<\infty} \frac{(c k)^k}{k\times k!}=T(c)-c&\text{for $c\leq 1/e$}, \\[2pt]
\infty &\text{for fixed $c>1/e$}.
\end{cases}
\label{limEZ}
\end{equation}
For the undirected complete graph, the expected number of undirected $c$-light $k$-cycles is $\frac12\E[Z^{(k)}_c]$ for $k\geq 3$, and the total expected number of $c$-light cycles is
\begin{equation}
\lim_{n\rightarrow\infty}\E[\text{\# undirected $c$-light cycles}]=
\begin{cases}
\displaystyle \frac{T(c)-c-c^2}{2}&\text{for $c\leq 1/e$}, \\
\infty &\text{for fixed $c>1/e$}.
\end{cases}
\label{limEZ-undirected}
\end{equation}
\end{theorem}
\begin{proof}
Equation~\eqref{eq:zk} is immediate from Lemmas~\ref{lemma:prelim} and~\ref{lemma:prelim2}.
For large $k$, by Stirling's formula 
the expression in~\eqref{eq:zk} is asymptotic to $$\frac{(c_2^k-c_1^k) e^k}{\sqrt{2\pi} k^{3/2}}.$$  Thus for fixed $c>1/e$, the expected number $Z_c$ of $c$-light cycles tends to $\infty$.  For $c\leq 1/e$, note that $\E[Z_c]\leq T(c)-c$, and that $\E[Z_c]\geq \sum_{k\leq k_0}\E[Z^{(k)}_c]$,
whose summands converge
to the $k\leq k_0$ terms for the series for $T(c)-c$.  Taking the $k_0\to\infty$ limit then yields~\eqref{limEZ}.
The formulas for the undirected complete graph follow immediately from the formulas for the directed complete graph.
\end{proof}

\section{Poisson approximation for short light cycles}

Next we show that the short $c$-light cycles are well approximated by
a Poisson process.  Here ``short'' means length at most
$L_0 = {\log n}/{(2\log\log n)},$
though for our main results it would suffice to prove this Poisson
approximation for any $L_0$ which tends to infinity as $n\to\infty$.
For $c\leq 1/e$, we know from the first-moment bounds in \tref{thm:EZ}
that with high probability there are no cycles of length $\omega(1)$.
It will then follow that the set of all $c$-light cycles is well
approximated by a Poisson process when $c\leq 1/e$.

For our purposes, the most convenient method to show Poisson
approximation is the Chen--Stein method, as formulated by Arratia,
Goldstein, and Gordon \cite[Theorem~2]{AGG}:
\begin{theorem}[\cite{AGG}] \label{thm:AGG}
  Let $\{ X_\alpha :\alpha\in I \}$ be a finite set of indicator random
  variables of dependent events, and let $\{Y_\alpha:\alpha\in I\}$ be a set of
  mutually independent Poisson random variables such that
  $\E[Y_\alpha]=\E[X_\alpha]$ for each $\alpha$.
  For each $\alpha$ let $B_{\alpha}$ be a subset of $I$, which is
  interpreted as the ``neighborhood of $\alpha$''.
Then the total variation distance between the dependent Bernoulli process
$(X_\alpha)_{\alpha\in I}$ and the independent Poisson process $(Y_\alpha)_{\alpha\in I}$ is at most
\begin{equation}
  2 \sum_{\alpha\in I}\,\sum_{\beta\in B_\alpha} \E[X_\alpha] \E[X_{\beta}]
 + 2 \sum_{\alpha\in I}\,\sum_{\substack{\beta\in B_\alpha\\ \beta\neq \alpha}} \E[X_\alpha X_{\beta}]
 + \sum_{\alpha\in I} \E\Big[\big|\E[X_\alpha | \{X_{\beta}: \beta\notin B_\alpha\}]-\E[X_\alpha]\big|\Big]
.
\label{Poisson-error}
\end{equation}
\end{theorem}

With a suitable choice of the neighborhood sets $B_\alpha$, the third
term above can easily be made zero, and analyzing the first two terms
above is manageable.
\newpage
\enlargethispage{10pt}
\begin{theorem} \label{short-Poisson}
  Suppose $c_0$ is fixed and $L_0= \log n /(2 \log\log n)$.
  For both the directed and undirected complete graphs with exponential edge weights,
  for any $\eps$, for
  sufficiently large $n$, the collection of $c_0$-light cycles with
  length at most $L_0$ is within total variation distance~$\eps$
  of a Poisson process whose intensity is the expected number of such cycles.
  In particular, except with probability~$\eps$, for all $k\leq L_0$ and all $c\leq c_0$,
   the number of $c$-light cycles of length $k$ equals the number of points
   in the corresponding region of the Poisson process.
\end{theorem}
\begin{proof}
  We divide the interval $(0,c_0]$ into subintervals of length
  $\Delta$.  To apply \tref{thm:AGG}, let $I$ denote the set of pairs
  $(C,c)$, where $C$ is a $k$-cycle (directed or undirected) with $k\leq L_0$, and
  $(c-\Delta,c]$ is one of the subintervals.  Let $X_{C,c}$ be the
  indicator random variable for cycle~$C$ being $c$-light (i.e., a
  cycle with mean weight at most $c/n$), but not $(c-\Delta)$-light.

  Let $B_{C,c}$ denote the subset of pairs $(C',c')\in I$ for which
  cycles $C$ and $C'$ have at least one edge in common.
  The variables $\{X_{C',c'}:
  (C',c')\notin B_{C,c}\}$ only depend on edges that are disjoint from
  cycle~$C$, so conditioning on them has no effect on the weight of
  cycle~$C$.  Thus
$$\E[X_{C,c} | \{X_{C',c'}: (C',c')\notin B_{C,c}\}]-\E[X_{C,c}]=0,$$
  and so the third term in \eqref{Poisson-error} is zero.

  Next we observe that the first two terms of \eqref{Poisson-error}
  are unaffected by the subdivision of the interval $(0,c_0]$: We can
  define $X_C=\sum_{c} X_{C,c}$, where the sum is over the right
  endpoints of the intervals in the subdivision of $(0,c_0]$, which
  are still Bernoulli random variables, and define $B_C$ to be the set
  of cycles that have at least one edge in common with $C$.  Then
  $\sum_C \sum_{C'\in B_C} \E[X_C]\E[X_{C'}] = \sum_{C,c}
  \sum_{(C',c')\in B_{C,c}} \E[X_{C,c}]\E[X_{C',c'}]$, and similarly for
  the second term.  Therefore we work with the $X_C$'s and $B_C$'s.

For the first term of \eqref{Poisson-error}, we write:
$$\sum_C\sum_{C'\in B_C} \E[X_C]\E[X_{C'}]=
  \sum_{k\leq L_0}\sum_{\ell \leq L_0}
  q_k q_{\ell}N_k N_{\ell} \Pr[\text{$k$-cycle and $\ell$-cycle share an edge}],
$$
where $q_k$ is the probability that a $k$-cycle is $c_0$-light.
The expected number of edge intersections between a $k$-cycle and an
$\ell$-cycle is $k \ell/n^2$ for the directed case, and $k\ell/\binom{n}{2}$ for the undirected case,
so they intersect with probability at
most $k \ell/\binom{n}{2}$.  Thus the first term is at most
$$ \frac{4}{n(n-1)} \left[\sum_{k\leq L_0} q_k N_k k\right]^2.$$
But from \tref{thm:EZ},
$q_k N_k \leq (c_0 e)^k/\sqrt{2\pi k^3}$.
So the first term of \eqref{Poisson-error} is bounded by
$$O\left(\frac{(c_0 e)^{2 L_0}}{n^2}\right)=\frac{1}{n^{2-o(1)}},$$
which tends to $0$ as $n\to\infty$.

For the second term of \eqref{Poisson-error},
we consider all possible pairs of distinct
non-edge-disjoint cycles $C,C'$ of $I$.  Let $k$ be the number of
edges common to $C$ and to $C'$, $k+\ell$ be the length of~$C$ and
$k+m$ be the length of~$C'$.  We let $w/n$ denote the total weight of
the edges $C$ and $C'$ share, $v/n$ denote the total weight of edges
in $C$ but not $C'$, and $x/n$ denote the total weight of edges in
$C'$ but not $C$.  The probability that both cycles are $c_0$-light is:
$$\E[X_C X_{C'}] =\iiint\limits_{\substack{w+v<(k+\ell) c_0\\w+x<(k+m)c_0}}
\frac{e^{-w/n} {(w/n)^{k-1}}}{{\Gamma(k)}}
\frac{e^{-v/n} {(v/n)^{\ell -1}}}{\Gamma(\ell)}
\frac{e^{-x/n} {(x/n)^{m-1}}}{\Gamma(m)}\frac{dw\,dv\,dx}{n^3}.
$$
We can bound the $e^{-w/n}$, $e^{-v/n}$, and $e^{-x/n}$ terms by $1$:
\begin{equation*}
\E[X_C X_{C'}] \leq \frac{1}{n^{k+\ell+m}}
\iiint\limits_{\substack{w+v<(k+\ell) c_0\\w+x<(k+m)c_0}}
\frac{w^{k-1}}{\Gamma(k)}
\frac{v^{\ell -1}}{\Gamma(\ell)}
\frac{x^{m-1}}{\Gamma(m)}\,dw\,dv\,dx.
\end{equation*}
We enlarge the domain of integration to the set of $(w,v,x)$ for which
$w<(k+\ell)c_0$, $v<(k+\ell)c_0$, and $x<(k+m)c_0$,
so that the triple integral has a product form that can be evaluated explicitly:
\begin{align*}
\E[X_C X_{C'}]
&\leq \frac{1}{n^{k+\ell+m}} \frac{((k+\ell)c_0)^{k}}{k!} \frac{((k+\ell)c_0)^{\ell}}{\ell!} \frac{((k+m)c_0)^{m}}{m!} \\
&\leq \frac{c_0^{k+\ell+m}}{n^{k+\ell+m}} (k+\ell)^{k+\ell} (k+m)^m \\
&\leq \frac{(c_0 L_0)^{2 L_0}}{n^{k+\ell+m}}.
\end{align*}

We now count the number of cycle pairs $(C,C')$ which are distinct and have at least one edge in common given $k,\ell,m$.

Suppose $C'\setminus C$ consists of $i\geq 1$ paths.  
There are at most $L_0^i$ possibilities for the lengths $m_1,m_2,\dots,m_i$ of the paths of $C'\setminus C$. 
With those lengths specified, we can list the $k+\ell$ vertices of $C$ in order from some arbitrary starting point, 
specify where along $C$ each path of $C'\setminus C$ starts and ends, and specify the $m_j-1$ vertices of each path. 
Thus, for either the directed or undirected setting, the number of such configurations is at most $n^{k+\ell}  L_0^i  L_0^{2i}  n^{m-i}$.

\enlargethispage{11pt}
Altogether the number of overlapping cycles $(C,C')$ is bounded by
$$\sum_{i=1}^{L_0} \frac{L_0^{3i}}{n^i}n^{k+\ell+m} \leq 2 \frac{L_0^3}{n} n^{k+\ell+m},$$
provided $L_0\leq\sqrt[3]{n/2}$.  There are at most $L_0$ choices for each of $k$, $\ell$, and $m$.  
The second term of \eqref{Poisson-error} is then bounded by $ 4 c_0^{2L_0} L_0^{2 L_0 + 6}/{n}$ (provided $L_0\leq\sqrt[3]{n/2}$).
When $L_0 = \frac12 \log n / \log\log n$, we have
\begin{align*}
\frac{4}{n} c_0^{2 L_0} L_0^{2 L_0 + 6} &\leq
\frac{4}{n} \exp[(\log\log n - \log\log\log n)(\log n / \log\log n + 6)+\log c_0 \log n /\log\log n] \\
&= 4 \exp[6 (\log\log n - \log\log\log n)\\
&\phantom{= 4 \exp[} - \log n \log\log\log n/ \log\log n + \log c_0 \log n /\log\log n],
\end{align*}
which tends to $0$ as $n\to\infty$.
\end{proof}

\section{Below the critical point: short light cycles}

Given the Poisson approximation result in \tref{short-Poisson} and the first-moment estimate in \tref{thm:EZ}, it is straightforward to derive the formulas for the mean-weight of the min mean-weight cycle (shown in \fref{cdf-4}), and the probability that the length of the cycle is $k$ for any fixed $k$ (in \tblref{tbl:p_k}).  Similar computations were done by Janson \cite{janson} for the min max-weight cycle.

\subsection{Weight of the cycle}

\begin{theorem} \label{thm:c}
For the directed complete graph, for fixed $c$,
there is a cycle with mean weight $\leq c/n$ with probability
$$
\lim_{n\to\infty} \Pr[\text{$\exists$ cycle with mean-weight $\leq c/n$}] =
\begin{cases} 1-\exp[-T(c)+c] & c\leq 1/e \\ 1 & c>1/e, \end{cases}$$
while for the undirected complete graph the probability is
$$
\lim_{n\to\infty} \Pr[\text{$\exists$ cycle with mean-weight $\leq c/n$}] =
\begin{cases} 1-\exp[(-T(c)+c+c^2)/2] & c\leq 1/e \\ 1 & c>1/e. \end{cases}$$
\end{theorem}
\begin{proof}
  For $c\leq 1/e$, by the first moment estimate, with probability
  $1-o(1)$ there are no $c$-light cycles with length $>L_0=\log
  n/(2\log\log n)$.  By the Poisson approximation, there is a
  $c$-light cycle of length $\leq L_0$ with probability
  $\exp[-\mu]+o(1)$, where $\mu=(1+o(1))\sum_{k=2}^{L_0} k^{k-1}
  c^k/k!=T(c)-c+o(1)$ for the directed complete graph, and
  $\mu=(T(c)-c-c^2)/2+o(1)$ for the undirected complete graph.
  For fixed $c>1/e$, the sum $(1+o(1))\sum_{k=2}^{L_0} k^{k-1} c^k/k!$
  tends to infinity with $n$, and the Poisson approximation still
  holds, so with probability $1-o(1)$ there is a $c$-light cycle.
\end{proof}

So the finiteness of $T(c)-c$ and $(T(c)-c-c^2)/2$ at $c=1/e$ accounts
for the discontinuities in the curves in \fref{cdf-4}.  Recalling the
behavior of the tree function near $c=1/e$, we see that these curves
for the min mean-weight cycle have a square-root plus constant
behavior to the left of the critical point.

\subsection{Length of the cycle}

\begin{theorem} \label{thm:pk}
  Suppose $k$ is fixed as $n\to\infty$.  For the directed complete graph, for $k\geq 2$
  \begin{align*}
  \lim_{n\to\infty} \Pr[\text{min mean-weight cycle has length $k$}] &=\\
  \lim_{n\to\infty} \Pr\left[\text{min mean-weight cycle has length $k$ and weight $\leq \frac{k}{e}$}\right] &=
  \int_0^{1/e} \frac{c^{k-1} k^k}{k!} e^{-T(c)+c}\,dc.
  \end{align*}
  For the undirected complete graph, for $k\geq 3$
  \begin{align*}
  \lim_{n\to\infty} \Pr[\text{min mean-weight cycle has length $k$}] &=\\
  \lim_{n\to\infty} \Pr\left[\text{min mean-weight cycle has length $k$ and weight $\leq \frac{k}{e}$}\right] &=
  \int_0^{1/e} \frac{c^{k-1} k^k}{2 k!} e^{(-T(c)+c+c^2)/2}\,dc.
  \end{align*}
\end{theorem}
\begin{proof}
We subdivide the interval $[0,1/e]$ into subintervals of width $\Delta$,
and let $[c,c+\Delta]$ be one of these subintervals.
By Theorem~\ref{thm:EZ}, the expected number of $k$-cycles which
are $(c+\Delta)$-light but not $c$-light is
\[
(1+o(1))\frac{(k c^{k-1} \Delta) k^k}{k\times k!} + O(\Delta^2)
\]
in the directed setting, and half that in the undirected setting (for $k\geq 3$),
where the $o(1)$ term goes to $0$ for fixed $k$ when $\Delta\to0$ and $n\to\infty$.
Using Poisson approximation for cycles of length at most $L_0=\log n /(2\log\log n)$
and mean weight $\leq 1/e$ (Theorem~\ref{short-Poisson}), the fact that it is unlikely that
there is any cycle with mean weight $\leq 1/e$ and length more than $L_0$ (Theorem~\ref{thm:EZ}),
and the probability that there is a cycle with mean weight $\leq c/n$ (Theorem~\ref{thm:c}),
we see that the probability that the min-mean weight cycle has length $k$ and weight between
$c$ and $c+\Delta$ is
\[
 (1+o(1)) \frac{c^{k-1} k^k}{k!}\times e^{-T(c)+c}\,\Delta  + O(\Delta^2)
\]
in the directed setting, and
\[
 (1+o(1))\frac12 \frac{c^{k-1} k^k}{k!}\times e^{(-T(c)+c+c^2)/2}\, \Delta  + O(\Delta^2)
\]
in the undirected setting,
where the $o(1)$ terms go to $0$ uniformly in $c$ for fixed $k$ when $n\to\infty$ and $\Delta\to0$.
Summing these expressions over the subintervals of $[0,1/e]$ and taking the $\Delta\to0$ limit
gives the integral expressions for the probability that the min mean-weight cycle has length $k$
and mean weight $\leq 1/e$.

Next we consider the possibility that the min mean-weight cycle has
length $k$ and mean weight $>1/e$.  Suppose $0<\delta<1$.  With
probability tending to $1$ as $n\to\infty$, there is a
$(1+\delta/k)/e$-light cycle.  But the expected number of $k$-cycles
that are $(1+\delta/k)/e$-light but not $1/e$-light tends to $0$ as
$\delta\to 0$.  So the probability that the min mean-weight cycle has
length $k$ and mean weight $>1/e$ tends to $0$ as $n\to\infty$.
\end{proof}

The formulas in \tblref{tbl:p_k} are rewritten slightly using Equation~\ref{eq:tree-rec} to write 
$e^{-T(c)} = {c}/{T(c)}.$
\tref{thm:c} and \tref{thm:pk} imply
$$\lim_{n\to\infty} \Pr[\text{min mean-weight cycle has mean weight $>1/e$}] > 0$$
but
$$\sum_k \lim_{n\to\infty} \Pr[\text{min mean-weight cycle has length $k$ and mean weight $>1/e$}] = 0.$$
There is no contradiction of course.  In Section~\ref{supercritical}
we further investigate the length of the min mean-weight cycle when its mean
weight is $>1/e$.

\subsection{Tail behavior of the length distribution}

We can approximate the large-$k$ behavior of the probability $p_k$
that the min mean-weight cycle has length $k$ (sending $n$ to infinity
first and then $k$).
We make the substitution $c=(1-\delta)/e$ to obtain, for the directed complete graph,
$$ p_k = \frac{e^{-k} k^k}{k!} e^{1/e} \int_0^{1} (1-\delta)^k \frac{e^{-\delta/e}}{T((1-\delta)/e)}\,\frac{d\delta}{e}.$$
The integrand is approximately $e^{-k\delta}$ for small $\delta$, and large $\delta$ contribute negligibly, 
so the integral is approximately $1/(k e)$, and so for large $k$
$$ p_k = (1+o(1))\frac{e^{-1+1/e}}{\sqrt{2\pi}} k^{-3/2}.$$
For the undirected complete graph, a similar computation yields
$$ p_k = (1+o(1))\frac{e^{-1/2+1/(2e)+1/(2 e^2)}}{2\sqrt{2\pi}} k^{-3/2}.$$

By comparison, Janson~\cite{janson} shows that for the min max-weight cycle on the
undirected complete graph, the expected number of cycles with max
weight at most $c/n$ is $\frac12( \log \frac{1}{1-c} - c - c^2/2)$, so
the probability that such a cycle exists is
$ 1-(1-c)^{1/2} e^{c/2+c^2/4}$
which has its threshold at $1$, and so
$$p_k = \frac{1}{2} \int_0^1 c^{k-1} (1-c)^{1/2} e^{c/2+c^2/4}\,dc,$$
which for large $k$ is
$$ p_k = (1+o(1)) \frac{\sqrt{\pi}}{4} e^{3/4} k^{-3/2}.$$

The computations for the min max-weight cycle on the directed complete graph are similar, though it is perhaps surprising that unlike the previous three cases, the asymptotics of $p_k$ in this case are $p_k=\Theta(k^{-2})$.  
More specifically, we have
$$p_k = \int_0^1 c^{k-1} (1-c) e^{c}\,dc,$$
Letting $c=1-\delta/k$, the integrand is for small $\delta$ asymptotically
$$ e^{-\delta}\frac{\delta}{k} e\, \frac{d\delta}{k} ,$$
so we see that the integral is asymptotically $e/k^2$.

\section{Above the critical point: long light cycles} \label{supercritical}

Recall that a cycle~$C$ is $c$-light if $\weight(C)\leq\length(C)
c/n$.  We say that a cycle~$C$ is \textit{$A$-uniformly $c$-light} if $C$ is
$c$-light, and in addition, for every subpath $P$ of $C$,
$$\weight(P)\leq \Big[\length(P) + A\Big]\frac{c}{n}.$$
In the directed complete graph,
let $Z_{L,\delta}$ denote the number of $(1+\delta)/e$-light cycles of
length between $L-1/\delta$ and $L$, and let $Y^{L_1,L_2}_{\delta,A}$ denote
the number of $A$-uniformly $(1+\delta)/e$-light cycles of length $L$ for
which $L_1<L\leq L_2$.  We will eventually choose the parameters
$\delta$, $L_1$, $L_2$, and $A$ so that
\begin{align*}
  \delta &= \Theta(1/\log^2 n)\\
  L_1, L_2 &= \omega(\log^2 n \log\log n) \\
  A &\approx \log n.
\end{align*}
We aim to show that with high probability such cycles exist.

We use a theorem of Koml\'os, Major, and Tusn\'ady \cite[Theorem~1]{KMT2}
from the second of two papers they wrote relating random walk to Brownian motion:
\begin{theorem}[\cite{KMT2}]\label{thm:KMT}
  Suppose that $X_1,X_2,\dots$ are i.i.d.\ random variables with expected value $0$, variance $1$, and finite exponential moments, i.e., their density function $f(x)$ satisfies $\int e^{t x} f(x)\,dx<\infty$ for $|t|\leq t_0>0$.  Then these random variables can be coupled to i.i.d.\ standard normal random variables $Y_1,Y_2,\dots$ such that for any $\lambda$ there are constants $K_1$ and $K_2$ for which
 $$ \Pr\left[\max_{k\leq n} \left|\sum_{i=1}^k X_i-\sum_{i=1}^k Y_i\right| > K_1\log n + x\right] < K_2\, e^{-\lambda x}.$$
\end{theorem}

\begin{lemma} \label{uniformly}
  Let $C$ be a particular cycle of length $L$, and let $W$ be its (random) weight.
  By definition, cycle $C$ is $(nW/L)$-light.
  The probability that $C$ is $A$-uniformly $(nW/L)$-light is independent of $W$ and $n$, and is at
  least $$\exp[-(\pi^2/2+o(1)) L/A^2],$$
  where the $o(1)$ term tends to $0$ as $L/A^2\to\infty$
  and $L/A^3\to 0$.
\end{lemma}
\begin{proof}
  The edge weights $W_1,\dots,W_L$ are distributed as $L$ independent
  exponential random variables with mean $1$.  The total weight is $W
  = \sum_{i=1}^L W_i$.  Because the edge weights are exponential
  random variables, conditioning on the total weight $W$ does not
  affect the joint distribution of relative weights $W_i/W$, which are
  distributed as the arc lengths of the arcs between $L$ uniformly
  random points placed on a circle of unit circumference.  The cycle
  is by definition $(nW/L)$-light.  
  The cycle is
  $A$-uniformly $(nW/L)$-light when
  \begin{equation} \label{W-uniform}
  \sum_{i=1}^\ell W_{a+i \bmod L} \leq (\ell+A) \frac{\sum_{i=1}^L W_i}{L}
  \end{equation}
  for each $a,\ell\in\{1,\dots,L\}$, where $a+i\bmod L$ is interpreted
  as a value in the range $1,\dots,L$.
  This property is invariant under uniform scalings of the edge weights,
  so whether or not the cycle is also
  $A$-uniformly $(nW/L)$-light is solely a function of these random
  arc lengths, independent of the total weight $W$ and $n$.

  Let $X_i=W_i-1$.  Suppose that $0<\eps\leq 1/4$ and
  \begin{equation} \label{Xtotal}
  0 \leq \sum_{i=1}^L X_i \leq \eps A,
  \end{equation}
 and that for each $k\in \{1,2,\dots,L\}$,
  \begin{equation}\label{Xpartial}
  -\frac{1-\eps}{2} A \leq \sum_{i=1}^k X_i \leq \frac{1-\eps}{2} A .
  \end{equation}
  Equation \eqref{Xtotal} implies $\sum_{i=1}^L W_i \geq L$.  When
  $a+\ell \leq L$, \eqref{Xpartial} implies $\sum_{i=1}^\ell W_{a+i}
  \leq \ell+(1-\eps)A$, which then implies \eqref{W-uniform} when
  $a+\ell\leq L$.  If $a+\ell>L$, then $\sum_{i=1}^\ell W_{a+i\bmod L} = \sum_{i=1}^L W_i - \sum_{i=1+a+\ell\bmod L}^{a} W_i \leq \eps A + (1-\eps)A + \ell$ by \eqref{Xtotal} and \eqref{Xpartial}, so again we have \eqref{W-uniform}.

  Now $X_i$ has zero mean, unit variance, and finite exponential
  moments, so Theorem~\ref{thm:KMT}
  implies that the partial sums $\sum_{i=1}^k X_i$ are well
  approximated by a standard Brownian motion.
  It is known how to compute using Fourier analysis the probability that a standard Brownian
  motion $B_t$ stays within an interval up through time $T$.  If the
  interval is $[-a,a]$, then this probability is
  $$\Theta\left(\exp\left[-\frac{\pi^2}{8} \frac{T}{a^2}\right]\right)$$
  (see e.g., \cite[pp.~216--218]{MR2604525}).
  Conditional upon the Brownian motion remaining within this interval,
  its final position within the interval at time $T$ has a well-behaved
  distribution which for large $T/a^2$ converges to a sine function, since the other Fourier coefficients decay more rapidly.

  We set $T=L$ and $a= A (1-2\eps)/2$; we see that the Brownian motion $B_t$ stays
  within the interval $\pm A(1-2\eps)/2$ and ends within the interval
  $(\eps A/3, 2\eps A/3)$ with probability at least
  \begin{equation} \label{intervals}
  \Theta\left(\eps \times
    \exp\left[-\frac{1}{(1-2\eps)^2}\frac{\pi^2}{2}
      \frac{L}{A^2}\right]\right).
  \end{equation}
  If this event occurs, and the
  partial sums $\sum_{i=1}^k X_k$ are within $\eps A/3$ of the
  Brownian motion, then equations \eqref{Xpartial} and \eqref{Xtotal}
  hold.

  By assumption $L/A^3\to 0$, so let us suppose $L/A^3\leq 1$.  Then
  $A \geq L^{1/3}$.  By assumption $L/A^2\to\infty$, so $L\to\infty$.
  Let us take
  $$ \eps = 6 \max\big(L/A^3, K_1\log L/A, e^{-L^{1/2}/A}\big), $$
  which by our assumptions tends to $0$, and we can suppose that it is
  at most $1/4$ as assumed above.

  In the KMT theorem, we choose $\lambda=20$ and $x=L/A^2$.  By our
  choice of $\eps$, the deviation $K_1 \log L + x$ is smaller than $\eps
  A/3$.  The probability that the Brownian motion and the random walk
  are not within $\eps A/3$ of one another is at most $K_2\, e^{-20
    L/A^2}$.  Now $20 > \pi^2/(2(1-2\eps)^2)$, so provided $L/A^2$ is
  sufficiently large, even if we condition on the unlikely event that
  the Brownian motion stays within the interval and ends within an even
  smaller interval, it is still extremely likely that the random walk
  does not deviate more than $\eps A/3$ from the Brownian motion.
  Thus, the probability that the cycle is $A$-uniformly $c$-light
  can be bounded below by an expression of the form \eqref{intervals}.
  Since $\eps \geq e^{-L^{1/2}/A}$ and $L/A^2\to\infty$, the factor of $\eps$ in
  \eqref{intervals} can be absorbed into the exponent where it becomes
  $\log\eps=o(L/A^2)$, and since $\eps\to0$, \eqref{intervals} can be
  further rewritten as $\exp[-(\pi^2/2+o(1)) L/A^2]$.
\end{proof}

\begin{lemma} \label{lemma:averagebig}
  For the directed complete graph, the expected number of $A$-uniformly $(1+\delta)/e$-light directed cycles
  with length more than $L_1$ and at most $L_2$ is
  \[\E\big[Y^{L_1,L_2}_{\delta,A}\big] = \sum_{L=L_1+1}^{L_2}
    \frac{(1+\delta)^L}{L^{3/2}}
    \exp[-(\pi^2/2+o(1)) L/A^2],\]
  where the $o(1)$ term tends to $0$ as $L_2^2/n\to 0$, $L_1/A^2\to\infty$,
  and $L_2/A^3\to 0$.  (For the undirected graph, the expected value is half as large.)
\end{lemma}
\begin{proof}
  Essentially immediate from the first moment estimates for $c$-light cycles (\tref{thm:EZ}) and
  \lref{uniformly}.  Since $L/A^2\to\infty$, the factor of $(1+o(1))/\sqrt{2\pi}$ gets absorbed into
  the $o(1)$ term in the exponent.
\end{proof}

\begin{lemma}\label{lemma:secondmoment}
  If $L_2^3 e^A/n\leq 1/2$, then 
\[\frac{\Var\big[Y^{L_1,L_2}_{\delta,A}\big]}{\E\big[Y^{L_1,L_2}_{\delta,A}\big]}\leq 
  1 + \frac{2 L_2^3 e^A (1+\delta)^{L_2}}{n}  .\]
  (This same bound holds for both the directed and undirected complete graph.)
\end{lemma}
\begin{proof}
  For any cycle~$C$ of length between $L_1$ and $L_2$, let $U_C$
  denote the indicator of the event that $C$ is uniformly light
  (within this proof, ``uniformly light'' means $A$-uniformly
  $(1+\delta)/e$-light).  Let $C_1$, $C_2$ be two cycles whose lengths
  are both in the range $(L_1,L_2]$.  If $C_1$
  and $C_2$ have no edges in common then $U_{C_1}$ and $U_{C_2}$ are
  independent.  So we have
\begin{align*}
\Var\big[Y^{L_1,L_2}_{\delta,A}\big]
&=\sum_{C_1}   \sum_{C_2} (\E[U_{C_1}U_{C_2}]-\E[U_{C_1}]\E[U_{C_2}]) \\
&=\sum_{C_1}\sum_{C_2: |C_2\cap C_1|\geq 1}  (\E[U_{C_1}U_{C_2}]-\E[U_{C_1}]\E[U_{C_2}]) \\
&\leq \sum_{C_1}\sum_{C_2: |C_2\cap C_1|\geq 1}  \E[U_{C_1}U_{C_2}] \\
&= \sum_{C_1}  \E[U_{C_1}]  \sum_{C_2: |C_2\cap C_1|\geq 1}  \E[U_{C_2}|U_{C_1}=1].
\end{align*}

\enlargethispage{22pt}
We partition the inner sum into sub-sums depending on the overlaps between $C_1$ and $C_2$.
If $C_2=C_1$, then of course $\E[U_{C_2}|U_{C_1}=1]=1$.  Otherwise,
$C_2\setminus C_1$ consists of a collection of disjoint paths, say that there are $i$ of them, and that their lengths are
$m_1,m_2,\ldots, m_i$.  Let $m=\sum_j m_j<L_2$.  To specify the $j$th path, we can specify its start and end points on $C_1$, as well as the internal vertices, so there are
$\leq L_2^2 n^{m_j-1}$ possible such subpaths.  Hence the number of such $C_2$'s is at most $L_2^{2i}n^{m-i}$.  Conditional on the weights in $C_1$, the probability that $C_2$ is uniformly light is at most the probability that for each $j$, the $j$th subpath $C_2\setminus C_1$ has weight at most $w_j:=(m_j+A)(1+\delta)/(en)$.  The probability that the $j$th subpath is light enough is at most
$$ \frac{w_j^{m_j}}{m_j!} \leq \frac{[(1+\delta) (1+A/m_j)]^{m_j}}{n^{m_j}\sqrt{2\pi m_j}}
  \leq  \frac{e^A (1+\delta)^{m_j}}{n^{m_j}},
$$
and so the conditional probability that $C_2$ is uniformly light is at most
$e^{A i}(1+\delta)^m/n^m$.
We see that the contribution to $\sum_{C_2: |C_2\cap C_1|\geq 1}  \E[U_{C_2}|U_{C_1}=1]$ from the case where $C_2\setminus C_1$
consists of $i$ paths of lengths $m_1,\dots,m_i$ is
\begin{align*}
\sum_{C_2: C_2\setminus C_1 =i \text{ paths of lengths }m_1,\dots, m_i}     \E[U_{C_2}|U_{C_1}]
  &\leq  \left[\frac{L_2^2 e^A}{n}\right]^i (1+\delta)^m.
\intertext{Since each $m_j\leq L_2$, the total contribution from cases where $C_2\setminus C_1$ consists of $i$ paths is}
\sum_{C_2: C_2\setminus C_1 =i\text{ paths}}
     \E[U_{C_2}|U_{C_1}] &\leq  \left[\frac{L_2^3 e^A}{n}\right]^i (1+\delta)^m.
\end{align*}
Since by assumption $L_2^3 e^A\leq n/2$,
we can sum over $i$ to obtain
\begin{align*}
\sum_{i\geq 1}\sum_{C_2:C_2\setminus C_1 =i \text{ paths}}  \E[U_{C_2}|U_{C_1}]
&\leq 2 \frac{L_2^3 e^A}{n} (1+\delta)^m.
\end{align*}
Combining this with the case $C_2=C_1$, and using $m\leq L_2$, we obtain
\[\sum_{C_2: |C_2\cap C_1|\geq 1}  \E[U_{C_2}|U_{C_1}=1]
 \leq 1 + 2 \frac{L_2^3 e^A}{n}(1+\delta)^{L_2}. \qedhere\]
\end{proof}

We now have the ingredients to prove our upper bound on the minimum mean weight:

\begin{theorem} \label{thm:max-weight}
  For both the directed and undirected complete graph, with probability
  $1-o(1)$ the mean weight of the minimum mean-weight cycle is at most
  $$\frac{\displaystyle 1+\frac{\pi^2/2+o(1)}{\log^2 n}}{en}.$$
\end{theorem}
\begin{proof}
We want to show that with high probability there are $(1+\delta)/e$-light cycles,
and we do this by showing that in fact there are $A$-uniformly $(1+\delta)/e$-light cycles
with high probability.  For any real-valued random variable $Y$ we have $\Pr[Y=0]\leq\Var[Y]/\E[Y]^2$.
We choose the parameters $A$, $L_1$, and $L_2$ so
that $\Var\big[Y^{L_1,L_2}_{\delta,A}\big]\ll \E\big[Y^{L_1,L_2}_{\delta,A}\big]^2$,
which will imply that with high
probability $Y^{L_1,L_2}_{\delta,A}>0$, i.e., that there is an $A$-uniformly $(1+\delta)/e$-light cycle with
size between $L_1$ and $L_2$.

We gather our constraints,
which are the same for both the directed and undirected complete graph:
\begin{align*}
L_1/A^2 & \gg 1\\
L_2/A^3 & \ll 1\\
\sum_{L=L_1+1}^{L_2}\frac{(1+\delta)^L}{L^{3/2}} \exp[-(\pi^2/2+o(1)) L/A^2]
 &\gg 1 + \frac{2 L_2^3 e^A (1+\delta)^{L_2}}{n} \\
\frac{2L_2^3 e^A}{n}&\leq 1\\
L_2^2 & \ll n
\end{align*}

Form the third
constraint above we need $\delta > (1+o(1))\pi^2/(2 A^2)$,
since otherwise the left-hand side would be smaller than one.
From the fourth constraint we need $A\leq\log n$,
so we need \[\delta\geq \frac{\pi^2+o(1)}{2\log^2 n}.\]
In order to make $\delta$ this small, we need $A=(1+o(1))\log n$ and $L_2\gg \log^2 n \log\log n$,
since otherwise the left-hand side of the third constraint would be smaller than one.
The second constraint then requires $L_2 \ll \log^3 n$.  Choosing $L_1$ and $L_2$ to include
many cycle lengths $L$ has the advantage of increasing $\E\big[Y^{L_1,L_2}_{\delta,A}\big]$, but this
increase is dwarfed by the exponential factors, and ultimately only affects the $o(1)$ term,
so we may as well take $L_1=L_2-1$.

We can make $\delta$ nearly as small as this bound by picking the following parameter values.
Let $\eps>0$
be an arbitrarily small positive constant.
\begin{align*}
A &= (1-\eps)\log n \\
\delta &= \frac{\pi^2/2 + 13\eps}{\log^2 n} \\
L_2 &= \frac{1}{\eps} \log^2 n \log\log n \\
L_1 &= L_2-1.
\end{align*}
It is straightforward to verify that for sufficiently small fixed
$\eps$, the above values satisfy the preceding constraints for all
sufficiently large $n$.
\end{proof}

The upper bound on the mean weight of the minimum mean-weight cycle is a key ingredient 
in bounding its length from below in the supercritical regime.

\begin{lemma}\label{lemma:averagesmall}
  If $\delta\ll 1$,
  then
  with high probability there are no cycles of length at most
\[ \frac{\log\delta^{-1}}{2\delta}\]
  that are $(1+\delta)/e$-light but not $1/e$-light.
  (This holds for both the directed and undirected complete graph.)
\end{lemma}

\begin{proof}
  Using the inequality version of the first-moment
  estimate~\eqref{eq:zk} and Stirling's formula, the expected number of cycles of length
  $\leq L_0$ that are $(1+\delta)/e$-light but not $1/e$-light is at most
  $$ \sum_{k=1}^{L_0} \frac{1} {\sqrt{2\pi} k^{3/2}}[(1+\delta )^k -1]. $$
  For $k\leq 1/\delta$ we can use the bound
  $$(1+\delta)^k - 1 \leq (e-1)\delta k,$$
  and find that the first $1/\delta$ terms (if there are that many) sum up to at most
  $$ \frac{e-1}{\sqrt{2\pi}} \int_0^{1/\delta} k^{-1/2}\delta\, dk = \frac{e-1}{\sqrt{\pi/2}}\,\delta^{1/2} = o(1),$$
  since by assumption $\delta\ll 1$.

  For the remaining terms we use the bound
  $$(1+\delta)^k - 1 \leq e^{\delta k},$$
  and group the terms into blocks of $\lceil 1/\delta\rceil$ terms (except that the last block may have fewer terms).
  For the block containing $k$, the block sum is $O(e^{\delta k} / (\delta k^{3/2}))$.  These
  block-sum bounds increase geometrically, so the total is
  $$O\left(\frac{e^{\delta L_0}}{\delta L_0^{3/2}}\right).$$
  Taking $L_0 = (\log\delta^{-1})/(2\delta)$, we have $e^{\delta L_0} = \delta^{-1/2}$,
  so the above bound is $O(1/(\delta L_0)^{3/2}) = O(1/(\log\delta^{-1})^{3/2})=o(1)$.
\end{proof}

We can now bound from below the length of the supercritical minimal mean-weight cycle:

\begin{theorem}
  For both the directed and undirected complete graph, conditional upon
  the min mean-weight cycle having mean weight $>1/(en)$, with probability
  $1-o(1)$ its length is at least
  $$(2/\pi^2-o(1)) \log^2 n \log\log n.$$
\end{theorem}
\begin{proof}
  Immediate from Theorem~\ref{thm:max-weight} and Lemma~\ref{lemma:averagesmall},
  together with the fact (Theorem~\ref{thm:c}) that the event being conditioned on
  has probability bounded away from $0$.
\end{proof}

\section*{Acknowledgements}
We thank Yuval Peres for helpful suggestions.  We also thank the referee and Oliver Riordin for helpful comments on the manuscript.

\phantomsection
\pdfbookmark[1]{References}{bib}
\bibliographystyle{hmralphaabbrv}
\bibliography{minmean}
\end{document}